\documentclass{amsart}

\usepackage{amssymb}

\usepackage{graphicx}

\usepackage{tikz,mathdots,enumerate}

\newtheorem{theorem}{Theorem}[section]
\newtheorem{lemma}[theorem]{Lemma}
\newtheorem{cor}[theorem]{Corollary}
\newtheorem{conjecture}[theorem]{Conjecture}
\newtheorem{prop}[theorem]{Proposition}

\theoremstyle{definition}
\newtheorem{definition}[theorem]{Definition}
\newtheorem{example}[theorem]{Example}

\theoremstyle{remark}
\newtheorem{remark}[theorem]{Remark}

\numberwithin{equation}{section}

\newcommand{\bR}{{\mathbb{R}}}

  \newcommand{\B}{{\mathcal{B}}}
  \newcommand{\C}{{\mathcal{C}}}
  \newcommand{\D}{{\mathcal{D}}}

  \newcommand{\G}{{\mathcal{G}}}
\renewcommand{\H}{{\mathcal{H}}}

\renewcommand{\P}{{\mathcal{P}}}

  \newcommand{\U}{{\mathcal{U}}}

	\newcommand{\ind}{\operatorname{ind}}

\begin{document}

\title{The generic rigidity of triangulated spheres with blocks and holes}


\author{J. Cruickshank}
\address{Sch. of Math., Stats. \& Appl. Math.,  
NUI Galway, Ireland.}
\curraddr{}
\email{james.cruickshank@nuigalway.ie}
\thanks{}

\author{D. Kitson}
\address{Dept.\ Math.\ \& Stats., Lancaster University,
Lancaster LA1 4YF \\U.K.
}
\curraddr{}
\email{d.kitson@lancaster.ac.uk}
\thanks{}

\author{S.C. Power}
\address{Dept.\ Math.\ \& Stats., Lancaster University,
Lancaster LA1 4YF, U.K.
}
\curraddr{}
\email{s.power@lancaster.ac.uk}
\thanks{} 

\subjclass[2010]{Primary 52C25 Secondary 05C75}

\date{}


\begin{abstract}
A simple graph $G=(V,E)$ is   $3$-rigid if its generic bar-joint frameworks in $\bR^3$ are infinitesimally rigid.
Block and hole graphs are derived from triangulated spheres by the removal of edges and the addition of minimally rigid subgraphs, known as blocks, in some of the resulting holes. Combinatorial characterisations of minimal $3$-rigidity are obtained for these graphs in the case of a single block and finitely many holes or a single hole and finitely many blocks. 
These results confirm a conjecture of Whiteley from 1988 and special cases of a stronger conjecture of Finbow-Singh and Whiteley from 2013. 
\end{abstract}

\maketitle



\section{Introduction}
A classical result of Cauchy \cite{cau} asserts that
a convex polyhedron in three-dimensional Euclidean space 
is continuously rigid, when viewed as a bar-joint framework, if and only if the faces are triangles. Dehn \cite{dehn} subsequently showed that this is also equivalent to the stronger condition of infinitesimal rigidity.
If the joints of such a framework are perturbed to  generic positions, with the bar lengths correspondingly adjusted, then infinitesimal rigidity may be established more directly by vertex splitting.  In this case convexity is not necessary and it follows that the graphs of triangulated spheres are \emph{$3$-rigid} in the sense that their generic placements in $\bR^3$ provide infinitesimally rigid bar-joint frameworks. This is a theorem of Gluck \cite{glu} and in fact these graphs are minimally $3$-rigid (isostatic) in view of their flexibility on the removal of any edge. The vertex splitting method was  introduced into geometric rigidity theory by Whiteley \cite{whi-vertex} and it plays a key role in our arguments.

While the general problem of characterising the rigidity or minimal rigidity of generic three-dimensional bar-joint frameworks remains open, an interesting  class of graphs which are derived from convex polyhedra has been considered in this regard by Whiteley \cite{whi-poly2}, Finbow-Singh, Ross and Whiteley \cite{fin-ros-whi} and Finbow-Singh and Whiteley \cite{fin-whi}. These graphs arise from surgery on a triangulated sphere involving the excision of the disjoint interiors of some essentially  disjoint triangulated discs and the insertion of minimally rigid blocks into some of the resulting holes. Even in the case of a single block and a single hole of the same perimeter length $n\geq 4$ the resulting block and hole graph need not be $3$-rigid. A necessary and sufficient condition for minimal rigidity for this  $n$-tower case, obtained in \cite{fin-whi}, requires that there exist $n$ internally disjoint edge paths connecting the vertices of one disc boundary to the vertices of the other boundary. 

\subsection{The main result}
In what follows we introduce some new methods which provide, in particular, characterisations of minimal $3$-rigidity for the class of block and hole graphs with a single block and finitely many holes. Such graphs may be viewed as the structure graphs of triangulated domes with windows, where the role of terra firma is played by the single block. 
In fact, the {\em girth inequalities},  defined in Sect.~\ref{Sect_GI},  provide a computable necessary and sufficient condition for $3$-rigidity in terms of lower bounds on the lengths of cycles of edges around sets of windows. 
 
\begin{center}
\begin{figure}[ht]
\centering
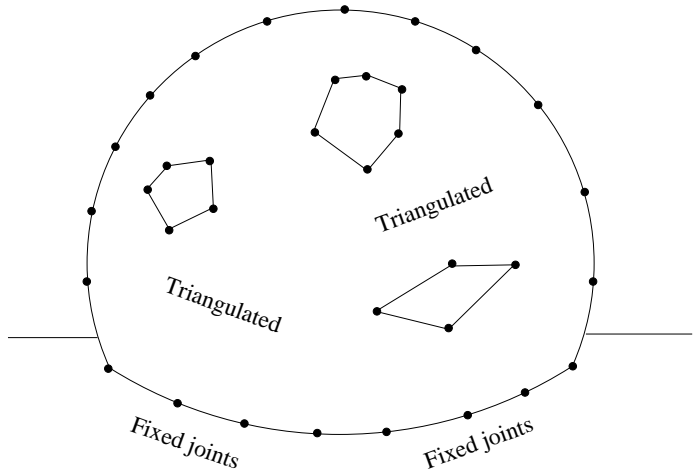
\caption{A triangulated dome with windows.}
\label{f:division}
\end{figure}
\end{center}

The main result is as follows. 

\begin{theorem}\label{t:oneblock}  
Let $\hat{G}$ be a block and hole graph with a single block and finitely many holes, or, a single hole and finitely many blocks. 
Then the following statements are equivalent.
\begin{enumerate}[(i)]
\item $\hat{G}$ is minimally $3$-rigid.
\item $\hat{G}$ is $(3,6)$-tight.
\item $\hat{G}$ is constructible from $K_3$ by the moves of vertex splitting and isostatic block substitution.
\item $\hat{G}$ satisfies the girth inequalities.
\end{enumerate}
\end{theorem}

Condition $(ii)$ is a well known necessary condition for minimally $3$-rigid graphs which requires the Maxwell count $|E|= 3|V|-6$ together with corresponding sparsity inequalities for subgraphs (see Sect.~\ref{Sect_BH}).
The construction scheme in $(iii)$ involves three phases of reduction for a $(3,6)$-tight block and hole graph, namely, 
\begin{enumerate}
\item
discrete homotopy reduction by $(3,6)$-tight preserving edge contractions, 
\item graph division over critical separating cycles of edges, and, 
\item admissible block-hole boundary contractions. 
\end{enumerate}
For the single block case, the equivalence of conditions $(i)-(iii)$ is established in Sect.~\ref{Sect_CD}.
The  girth inequalities (see Sect.~\ref{Sect_GI}) are a reformulation of the \emph{cut cycle inequalities} of \cite{fin-whi}. 
The same equivalences are obtained for the ``dual" class of block and hole graphs with a single hole in Theorem \ref{t:onehole}. In fact, the dual of \emph{any} generically isostatic block and hole graph is  generically isostatic (see \cite{fin-ros-whi}).

Theorem \ref{t:oneblock} confirms the single hole case and the single block case of Conjecture 5.1 in \cite{fin-whi} (see also Remark \ref{Rem_FSW} below). Example \ref{CounterEx} shows that the conjecture is not true in general. A further corollary of Theorem \ref{t:oneblock}  is that  the following conjectures, paraphrased from \cite{whi-poly2}, are true. 

\begin{conjecture}[{\cite[Conjectures 4.2 and 4.3]{whi-poly2}}]
\label{Conj_WW}
Let $\hat{G}$ be a block and hole graph with one pentagonal block 
and two quadrilateral holes, or, two quadrilateral blocks and one pentagonal hole. If $\hat{G}$ is $5$-connected then it is minimally $3$-rigid.
\end{conjecture}

The Appendix provides a proof of the preservation of minimal $3$-rigidity under vertex splitting (established in \cite{whi-vertex}) and a simple proof of Gluck's theorem (\cite{glu}) on the $3$-rigidity of graphs of triangulated spheres.

\section{Block and hole graphs} 
\label{Sect_BH}
A cycle of edges in a simple graph is a sequence 
$e_1,e_2,\dots , e_r,$ with $r \geq 3$, for which
there exist vertices $v_1, v_2,\dots, v_r$, such that $e_i = v_iv_{i+1}$ for $i<r$ and $e_r=v_rv_1$.
A cycle of edges is \emph{proper} if its vertices are distinct.

\subsection{Face graphs}
Let $S=(V,E)$ be the graph of a triangulated sphere, that is, $S$ is a planar simple $3$-connected graph such that each face of $S$ is bounded by a  $3$-cycle. Let $c$ be a proper cycle in $S$ of length four or more. Then $c$ determines two complementary planar  subgraphs of $S$, each with a single non-triangular face bordered by the edges of $c$. Such a subgraph is referred to as a \emph{simplicial disc}  of $S$ with boundary cycle $c$.
The boundary cycle of a simplicial disc $D$ is also denoted by $\partial D$. 
The \emph{edge interior} of $D$ is the set of edges in $D$ that do not belong to $\partial D$. 
A collection of simplicial discs is \emph{internally-disjoint} if their respective edge interiors are pairwise disjoint. 

\begin{definition}
A {\em face graph}, $G$, is obtained from the graph of a triangulated sphere, $S$, by,
\begin{enumerate}
\item choosing a collection of internally disjoint simplicial discs in $S$, 
\item removing the edge interiors of each of these simplicial discs,
\item labelling the non-triangular faces of the resulting planar graph by either $B$ or $H$.
\end{enumerate}
\end{definition}

A labelling of the triangular faces of $G$ by the letter $T$ would be redundant but nevertheless an edge of $G$ is said to be
of type $BB, BH, HH, BT, HT$ or $TT$ according to the labelling of its adjacent faces.
A face graph is of type $(m,n)$ if the number of $B$-labelled faces is $m$ and the number of $H$-labelled faces is $n$.

\begin{example}
The complete graph $K_4$ is the graph of a triangulated sphere and may be expressed as the union of two simplicial discs with a common $4$-cycle boundary. The edge interiors of these simplicial discs each contain a single edge. Remove these edge interiors to obtain a $4$-cycle and label the two resulting faces by $B$ and $H$. This is the smallest example of a face graph of type $(1,1)$.
\end{example}

If  $\B$ and $\H$  are collections of simplicial discs of $S$ then the notation $G=S(\B,\H)$ indicates that the $B$-labelled faces of the face graph $G$ correspond to the simplicial discs in $\B$ and the $H$-labelled faces of $G$ correspond to the simplicial discs in $\H$.

\subsection{Block and hole graphs}
Let $G=S(\B,\H)$ be a face graph derived from $S$ and let $\B=\{B_1,B_2,\ldots,B_m\}$ be the simplicial discs in $S$ which determine the $B$-labelled faces of $G$.





\begin{definition}
A {\em block and hole graph} on $G=S(\B,\H)$ 
is a graph $\hat{G}$ of the form $\hat{G}=G\cup \hat{B}_1 \cup \cdots \cup \hat{B}_m$ where,
\begin{enumerate}
\item $\hat{B}_1,\hat{B}_2, \ldots, \hat{B}_m$ are minimally $3$-rigid graphs which are either pairwise disjoint, or, intersect at vertices and edges of $G$,
\item
$G\cap \hat{B}_i = \partial B_i$ for each $i=1,2,\ldots,m$.
\end{enumerate}
\end{definition}

As in \cite{fin-whi,fin-ros-whi}, we refer to the subgraphs $\hat{B_i}$  as the \emph{blocks} or \emph{isostatic blocks} of $\hat{G}$.
The following isostatic block substitution principle asserts that one may substitute isostatic blocks without altering the rigidity properties of $\hat{G}$. The proof is an application of \cite[Corollary 2.8]{whi-poly1}.


\begin{lemma}\label{l:graphequivalentsB}
Let $G=S(\B,\H)$ be a face graph and suppose there exists a block and hole graph on $G$ which is simple and minimally $3$-rigid. 
Then every simple block and hole graph on $G$ is minimally $3$-rigid. 
\end{lemma}


The graph of a triangulated sphere is minimally $3$-rigid (\cite{glu}) and so such graphs provide a natural choice for the isostatic blocks in a block and hole graph.  

\begin{example}
\label{Ex2}
Let $G=S(\B,\H)$ be a face graph and for each $B_i\in \B$ construct an isostatic block $B_i^\dagger$ with,
\[
V(B_i^\dagger) = V(\partial B_i) \cup\{x_i,y_i\},
\quad E(B_i^\dagger) =E(\partial B_i) \cup \{(v,x_i),(v,y_i):v \in  V(\partial B_i)\}
\]
The graph $B_i^\dagger$ is referred to as a {\em simplicial discus} with {\em poles} at $x_i$ and $y_i$. 
The resulting block and hole graph $G\cup B_1^\dagger \cup \cdots \cup B_m^\dagger$, denoted by $G^\dagger$, is  referred to as the {\em discus and hole graph} for $G$. Note that
$G^\dagger$ is  a simple graph which is uniquely determined by $G$.
The discus and hole graphs will be used in Sect.~\ref{Sect_CD} to establish a construction scheme for $(3,6)$-tight block and hole graphs with a single block. 
\end{example}

In general, a block and hole graph may not be simple. This can occur if two $B$-labelled faces of $G$ share a pair of non-adjacent vertices.

\begin{example}
\label{Ex1}
Let $G=S(\B,\H)$ be a face graph and for each $B_i\in \B$ construct an isostatic block $B_i^\circ$ as follows:
Define $B_i^\circ$ to be the graph of a triangulated sphere which is obtained from the boundary cycle $\partial B_i$ by adjoining $2(|\partial B_i|-3)$ edges so that $B_i^{\circ}$ is the union of two internally-disjoint simplicial discs with common boundary cycle $\partial B_i$. The resulting block and hole graph $G\cup B_1^{\circ} \cup \cdots \cup B_m^{\circ}$ will be denoted $G^{\circ}$.
Note that $G^{\circ}$ is not uniquely determined and may not be simple. However, $G^{\circ}$ has the convenient property that its vertex set is that of $G$. This construction will be applied in Sect.~\ref{Sect_GI} to characterise isostatic block and hole graphs in terms of girth inequalities. 
\end{example}


There is a simple relationship between a face graph $G$ and its associated block and hole graphs. It is convenient therefore to focus the reduction analysis at the level of face graphs.
This perspective also underlines a duality principle of the theory under $B$, $H$ transposition, a feature exposed in \cite{fin-ros-whi} and discussed in Sect. \ref{BHTrans}.


\subsection{Freedom numbers}
Let $f(J)$  denote the \emph{freedom number} $3|V(J)|-|E(J)|$ of a graph $J$. A simple graph $J$ {\em satisfies the Maxwell count} if 
$f(J)=6$. 

\begin{lemma}
\label{MaxSubs}
Let $G$, $K$ and $K'$ be graphs with the following properties,
\begin{enumerate}[(i)]
\item $K$ and $K'$ both satisfy the Maxwell count, and,
\item $G\cap K = G\cap K'$.
\end{enumerate}
If $G\cup K$ satisfies the Maxwell count 
then $G\cup K'$ satisfies the Maxwell count. 
\end{lemma}

\proof 
The result follows on considering the freedom numbers,
\[f(G\cup K')=f(G)+f(K')-f(G\cap K')
=f(G)+f(K)-f(G\cap K) = f(G\cup K)=6.\]
\endproof

A simple graph $G$ is said to be  {\em $(3,6)$-sparse} if $f(J)\geq 6$ for any subgraph $J$ containing at least two edges. The graph $G$ is {\em $(3,6)$-tight} if it is $(3,6)$-sparse and satisfies the Maxwell count.


\begin{lemma}
\label{CombSubs}
Let $G$, $K$ and $K'$ be simple graphs with the following properties,
\begin{enumerate}[(i)]
\item $K$ and $K'$ are both $(3,6)$-tight,
\item $G\cap K = G\cap K'$,
\item if $v,w\in V(G\cap K')$ and $vw\in E(K')$ then $vw\in E(G)$.
\end{enumerate}
If $G\cup K$ is $(3,6)$-sparse (respectively, $(3,6)$-tight) then $G\cup K'$ is $(3,6)$-sparse (respectively, $(3,6)$-tight). 
\end{lemma}

\proof 
Suppose that $G\cup K$ is $(3,6)$-sparse and let $J$ be a subgraph of $G\cup K'$ which contains at least two edges. It is sufficient to consider the case where $J$ is connected. 
If $J$ is a subgraph of $G$ then $f(J)\geq 6$ since $G\cup K$ is $(3,6)$-sparse. If $J$ is not a subgraph of $G$ then there are two possible cases.

{\em Case 1)} Suppose that $J\cap K'$ contains exactly one edge $vw$ and that this edge is not in $G$. Then, by condition $(iii)$, either $v\notin V(G)$ or $w\notin V(G)$.
It follows that, 
\[f(J) = f(J\cap G) + (f(J\cap K') - f(J\cap (G\cap K'))) \geq 5+2=7.\]

{\em Case 2)} 
Suppose  that $J\cap K'$ contains two or more edges.  
Since $K$ satisfies the Maxwell count, 
$f(J\cap K')\geq 6=f(K)$ and, since $G\cup K$ is $(3,6)$-sparse, 
\begin{eqnarray*}
f(J) &=& f(J\cap G) + f(J\cap K') - f(J\cap (G\cap K')) \\
&\geq& f(J\cap G) + f(K) - f(J\cap (G\cap K)) \\
&=& f((J\cap G)\cup K) 
\geq 6.
\end{eqnarray*}
In each case, $f(J)\geq 6$ and so $G\cup K'$ is $(3,6)$-sparse.
If $G\cup K$ is $(3,6)$-tight then by the above argument, and Lemma
\ref{MaxSubs}, $G\cup K'$ is also $(3,6)$-tight.
\endproof

It is well-known that minimally $3$-rigid graphs, and hence the isostatic blocks of a block and hole graph, are necessarily $(3,6)$-tight (see for example \cite{gra-ser-ser}). 
The following corollary refers to the discus and hole graph described in Example \ref{Ex2}.

\begin{cor}\label{l:graphequivalentsA}
Let $G=S(\B,\H)$ be a face graph of type $(m,n)$.
\begin{enumerate}[(i)]
\item Suppose there exists a block and hole graph on $G$ which satisfies the Maxwell count. 
Then every block and hole graph on $G$ satisfies the Maxwell count. 

\item Suppose there exists a block and hole graph on $G$ which is simple and $(3,6)$-sparse (respectively, simple and $(3,6)$-tight). 
Then the discus and hole graph $G^\dagger$ is $(3,6)$-sparse
(respectively, $(3,6)$-tight). 
\end{enumerate}
\end{cor}

\proof
The statements follow by applying Lemmas \ref{MaxSubs} and \ref{CombSubs} respectively with $K$ and $K'$ representing two different choices of isostatic block for a given $B$-labelled face of $G$.
Note that in the case of $(ii)$, if $B_i\in \B$ then 
there are no edges $vw$ of the simplicial discus $B_i^\dagger$ with $v,w\in \partial B_i$ other than the edges of the boundary cycle $\partial B_i$.
Thus condition $(iii)$ of Lemma \ref{CombSubs} is satisfied. 
\endproof

\subsection{$3$-connectedness}
Recall that a graph is $3$-connected if there exists no pair of vertices $\{x, y\}$ with the property that there are two other vertices which cannot be connected by an edge path avoiding $x$ and $y$. Such a pair is referred to here as a separation pair.
The block and hole graphs  $\hat{G}$ which are derived from face graphs $G$ need not be $3$-connected.
However, it is shown below that in the single block case $3$-connectedness is a consequence of $(3,6)$-tightness.    

\begin{lemma}
\label{3Connected}
Every $(3,6)$-tight block and hole graph with a single block is $3$-connected.
\end{lemma}

\begin{proof}
Let $\hat{G}$ be a $(3,6)$-tight block and hole graph with a single block and suppose that  $\hat{G}$ is not $3$-connected. Then there exists a separation pair $\{x, y\}$ with edge-connected components $K_1, K_2, \dots, K_r$. Let $K_1$ be the component which contains an edge of $\hat{B_1}$ and hence all of  $\hat{B_1}$. The graph
$K_1$ and its complementary graph $K_1'$ with $E(K')=E(\hat{G})\backslash E(K)$ each have more than one edge
and their intersection is $\{x, y\}$. Thus $f(K_1\cap K_1')=6$ and  
\[
f(K_1)+f(K_1')=f(\hat{G})+f(K_1\cap K_1')=12
\]
It follows that the $(3,6)$-sparse graphs  $K_1$ and $K_1'$ are both $(3,6)$-tight. In particular, $K_1'$ must be the graph of a triangulated sphere and it follows that $K_1'$ contains the edge $xy$. Now $K_1\cup \{xy\}$ is a subgraph of $\hat{G}$ which fails the $(3,6)$-sparsity count, which is a contradiction. 
\end{proof}



\begin{remark}
\label{Rem_FSW}
The definition of a block and hole graph $\hat{G}$ is somewhat more liberal than the block and hole graphs $\hat{\P}$ of Finbow-Singh and Whiteley \cite{fin-whi}. 
A graph $\hat{\P}$ is defined by considering a planar $3$-connected graph $\P$ whose faces are labelled with the letters $B, H$ and $D$. The $B$-labelled faces are replaced with isostatic block graphs and the $D$-labelled faces are triangulated. The resulting graph $\hat{\P}$ is called a \emph{base polyhedron} reflecting the fact that it is the starting point for an ``expanded" graph $\hat{\P}^E$. This is obtained by a  further triangulation process involving adding vertices on edges of $DD$ type, and vertices interior to triangles. In particular  $\hat{\P}$ and $\hat{\P}^E$  are also $3$-connected. 
\end{remark}

\section{Edge contraction and cycle division}
\label{Sect_CD}
For $m,n$ nonnegative integers let $\G(m,n)$ be the set of all face graphs of type $(m,n)$ for which the discus and hole graph $G^\dagger$ is $(3,6)$-tight. In particular, the graphs of $\G(0,0)$
are triangulations of a triangle and the sets $\G(0,n)$ and $\G(m,0)$
are empty for $n,m \geq 1$.

\subsection{$TT$ edge contractions}
The first reduction move for block and hole graphs is based on an edge contraction move for face graphs.
A $TT$ edge in a face graph $G$ is said to be \emph{contractible} if it belongs to two triangular faces and to no other $3$-cycle of $G$. In this case  the deletion of the edge and the  identification of its vertices  determines a graph move $G \to G'$ on the class of face graphs, called a {\em TT edge contraction}, which preserves the boundary cycles of the labelled faces of $G$. 


\begin{definition}
A \emph{terminal face graph} $G$ in  $\G(m,n)$ is one for which there exist no $TT$ edge contractions $G \to G'$ with $G' \in \G(m,n)$. 
\end{definition}

\begin{example}
A cycle graph with length at least $4$, with exterior face labelled $B$ and interior face labelled $H$ is evidently a terminal graph in $\G(1,1)$. 
\end{example}

\begin{example}
Fig.~\ref{f:5holebase} shows a face graph $G$ with a contractible $TT$ edge which is nevertheless a terminal face graph of $\G(1,5)$. The discus and hole graph for the contracted graph $G'$ 
fails to be $(3,6)$-tight since there is an extra edge added to
the simplicial discus $B^\dagger$. Each block and hole graph $\hat{G}$  is evidently reducible by inverse Henneberg moves to a single block (i.e. by successively removing degree $3$ vertices, see for example \cite{gra-ser-ser}). However, there is a systematic method of reduction described below in which each move is a form of edge contraction or cycle division.
\end{example}

\begin{center}
\begin{figure}[ht]
\centering
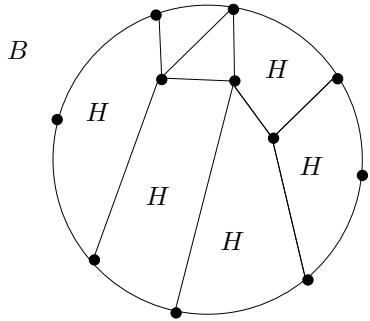
\caption{A terminal face graph in  $\G(1,5)$.}
\label{f:5holebase}
\end{figure}
\end{center}

\begin{example}
The $6$-vertex graph of Fig.~\ref{f:DBgraph} is a terminal face graph in $\G(2,2)$ whose block and hole graphs (variants of the double banana graph) are not $3$-rigid. The graph $G^\circ$ (see Ex.~\ref{Ex1}), which in this case is unique, is not a simple graph.
\end{example}

\begin{center}
\begin{figure}[ht]
\centering
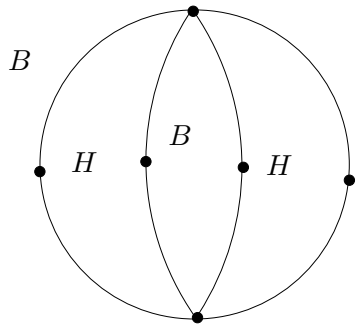
\caption{A terminal face graph in $\G(2,2)$.}
\label{f:DBgraph}
\end{figure}
\end{center}

\begin{remark}
The contraction of a $TT$ edge in a graph which is both $(3,6)$-tight and $3$-connected  may remove either one of these properties while maintaining the other. However, for a block and hole graph with a single block the situation is more straightforward since, by Lemma 
\ref{3Connected}, $3$-connectedness is a consequence of $(3,6)$-tightness. In particular, if $G$ is a terminal face graph in $\G(1,n)$, for some $n\geq 1$,  then the discus and hole graph $G^\dagger$ is both $(3,6)$-tight and $3$-connected.  
\end{remark}

\subsection{Critical separating cycles}
Let $c$ be a proper cycle of edges in a face graph $G$ and fix a planar realisation of $G$. Then $c$ determines two new face graphs $G_1$ and $G_2$ which consist of the edges of $c$ together with the edges and labelled faces of $G$ which lie outside 
(resp. inside) $c$.  
If $c$ is not a $3$-cycle then the face in $G_1$ (and in $G_2$) which is bounded by $c$ is assigned the label $H$. 
The discus and hole graph for $G_1$ (resp. $G_2$) will be denoted  $Ext(c)$ (resp. $Int(c)$).
Note that $G^\dagger = Ext(c)\cup Int(c)$ and $Ext(c)\cap Int(c) = c$.

\begin{definition}\label{d:critcycle}
A {\em critical separating cycle} for a face graph $G$ 
is a proper cycle $c$ with the property that either $Ext(c)$ or $Int(c)$ is $(3,6)$-tight.
\end{definition}

The boundary of a $B$-labelled face is always a critical separating cycle. Moreover, if $G^\dagger$ is $(3,6)$-tight then the boundary of every face of $G$ is a critical separating cycle. 

\begin{lemma}
\label{CritSep3Cycle}
Let $G$ be a face graph in $\G(m,n)$. 
If $c$ is a $3$-cycle in $G$ then $c$ is a critical separating cycle for $G$ and both $Ext(c)$ and $Int(c)$ are $(3,6)$-tight.
\end{lemma}

\proof
Since $G^\dagger$ is $(3,6)$-sparse, both $Ext(c)$ and $Int(c)$ are $(3,6)$-sparse. Note that $f(G^\dagger) = f(c) = 6$,
$f(Ext(c))\geq 6$ and $f(Int(c))\geq 6$. 
Thus applying the formula,
\[f(G^\dagger) = f(Ext(c)) + f(Int(c)) - f(c),\] 
it follows that both $Ext(c)$ and $Int(c)$ are $(3,6)$-tight.
\endproof




For face graphs of type $(1,n)$ a planar depiction may be chosen for which the unbounded face is $B$-labelled. Thus for any proper cycle, it may be assumed that $Ext(c)$ contains the isostatic block and $Int(c)$ is a subgraph of a triangulated sphere.

\begin{lemma}
\label{CritSep2}
Let $G$ be a face graph in $\G(1,n)$.
Then  a proper cycle $c$ is a critical separating cycle for $G$
if and only if $Ext(c)$ is $(3,6)$-tight.
\end{lemma}

\proof
If $c$ is a $3$-cycle then apply Lemma \ref{CritSep3Cycle}.
If $c$ is not a $3$-cycle then $Int(c)$ is a subgraph of a triangulated sphere with $f(Int(c))\geq 6+(|c|-3)>6$. 
\endproof





\begin{prop}
\label{FaceProp}
Let $G$ be a face graph of type $(1,n)$ and suppose that there are no $TT$ or $BH$ edges in $G$. 
\begin{enumerate}[(i)]
\item
If $G^\dagger$ satisfies the Maxwell count then
$G$ contains a proper cycle $\pi$, which is not the boundary of a face, such that $Ext(\pi)$ satisfies the Maxwell count.
\item
If $G\in \G(1,n)$ then $G$ contains a critical separating cycle for $G$ which is not the boundary of a face.
\end{enumerate}
\end{prop}

\proof
Since $G$ contains no edges of type $TT$ or $BH$, every edge in the boundary cycle $\partial B$ is of type $BT$ (see Fig.~\ref{f:tooth}) and so each vertex $v$ in $\partial B$ must be contained in an $H$-labelled face $H_v$. 
If each vertex $v$ in $\partial B$ is contained in a distinct $H$-labelled face $H_v$ then let $r=|\partial B|$ and let $v_1,\ldots,v_r$ be the vertices of $\partial B$. Let $H_1,\dots,H_n$ be the $H$-labelled faces of $G$, indexed so that $H_i=H_{v_i}$ for each $i=1,2,\ldots,r$. Note that $r\leq n$.  
Since the block and hole graphs $G^\circ$ satisfy the Maxwell count it follows that,
\[
r-3=|\partial B|-3 = \sum_{i=1}^n(|\partial H_i|-3) \geq \sum_{i=1}^r(|\partial H_i|-3) \geq r
\]
This is a contradiction and so $H_v=H_w$ for some distinct vertices $v,w \in\partial B$. 
The boundary of this common $H$-labelled face is composed of two edge-disjoint paths $c_1$ and $c_2$ joining $v$ to $w$.
The boundary cycle $\partial B$ is also composed of two edge-disjoint paths joining $v$ to $w$. Let $\pi_1$ be the path in Fig.~\ref{f:tooth} which moves anti-clockwise along $\partial B$ from  $v$ to $w$ and then along $c_1$ from $w$ to $v$.
Similarly, let $\pi_2$ be the path which moves clockwise along $\partial B$ from  $v$ to $w$ and then along $c_2$ from $w$ to $v$.
Note that $\pi_1$ and $\pi_2$ are proper cycles in $G$ with
$Ext(\pi_1)\cap Ext(\pi_2) = B^\dagger$. Thus, 
\[f(G^\dagger)=f(Ext(\pi_1))+f(Ext(\pi_2))-f(B^\dagger),\]
and so, since $f(G^\dagger)=f(B^\dagger)=6$, it follows that
$f(Ext(\pi_1))=f(Ext(\pi_2))=6$.
Hence $Ext(\pi_1)$ and $Ext(\pi_1)$ both satisfy the Maxwell count.
This proves $(i)$ and now $(ii)$ follows immediately.
\endproof

\begin{center}
\begin{figure}[ht]
\centering
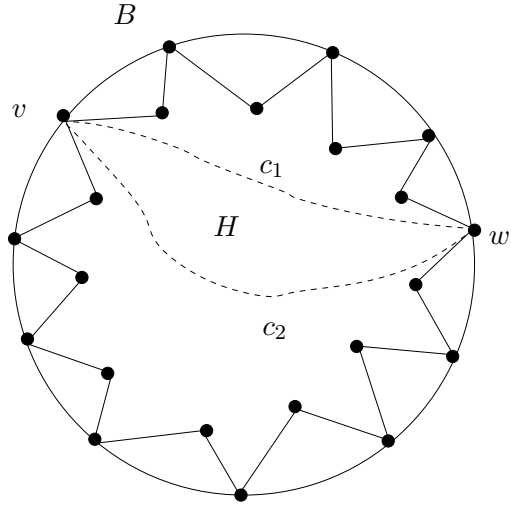
\caption{$H=H_v=H_w$.}
\label{f:tooth}
\end{figure}
\end{center}

\subsection{Separating cycle division}
The next reduction move for block and hole graphs is based on a division of the face graph with respect to  a critical separating cycle of edges. The usefulness of this arises from the fact that critical separating cycles arise when there are obstructions to $TT$ edge contraction.

\begin{definition}
Let $G$ be a face graph with a single $B$-labelled face and consider a planar realisation in which the unbounded face is labelled by $B$. Let $c$ be a proper cycle in $G$. 

Define $G_1$ to be the face graph obtained from $G$ and $c$ by,
\begin{enumerate}[(i)]
\item removing all edges and vertices interior to $c$, and,
\item if $|c|\geq4$, viewing the edges of $c$ as the boundary of a new face with label  $H$. 
\end{enumerate}

Define $G_2$ to be the face graph obtained from $G$ and $c$ by,
\begin{enumerate}[(i)]
\item  removing all edges and vertices which are exterior to $c$, and, 
\item if $|c|\geq4$, viewing the edges of $c$ as the boundary of a new face with label  $B$. 
\end{enumerate}
This division process $G \to \{G_1, G_2\}$ is referred to as a \emph{separating cycle division} for the face graph $G$ and cycle $c$. 
\end{definition}

Note that, under this separating cycle division, $G_1^\dagger = Ext(c)$. If $|c|=3$ then $G_2^\dagger=Int(c)$ while if $|c|\geq 4$ then $G_2^\dagger=Int(c)\cup B^\dagger$ where $B^\dagger$ is the simplicial discus with perimeter vertices in $c$.

\begin{center}
\begin{figure}[ht]
\centering
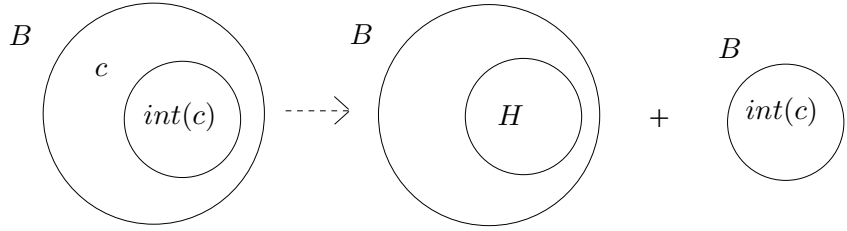
\caption{Separating cycle division in a face graph.}
\label{f:division}
\end{figure}
\end{center}

\begin{lemma}
\label{CycleDivision}
Let $G$ be a face graph in $\G(1,n)$ with a separating cycle division $G \to \{G_1, G_2\}$ for a critical separating cycle $c$ in $G$.
\begin{enumerate}[(i)]
\item If $|c|=3$ then $G_1\in \G(1,n)$ and $G_2 \in \G(0, 0)$.
\item 
If $|c|\geq 4$ then  $G_1\in \G(1,n-l+1)$ and $G_2 \in \G(1, l)$, where $l$ is the number of $H$-labelled faces interior to $c$. 
\end{enumerate}
\end{lemma}

\begin{proof}
$(i)$ By Lemma \ref{CritSep3Cycle}, $G_1$ and $G_2$ both have $(3,6)$-tight discus and hole graphs. Since $G_2$ has no $B$-labelled faces it must be the graph of a triangulated sphere. 

$(ii)$
By Lemma \ref{CritSep2}, $G^\dagger_1=Ext(c)$ is $(3,6)$-tight.
That $G^\dagger_2$ is $(3,6)$-tight follows from Lemma \ref{CombSubs} since $G^\dagger=Ext(c)\cup Int(c)$ is $(3,6)$-tight and $Ext(c)$ (which intersects $Int(c)$ in $c$) may be substituted by the simplicial discus $B^\dagger$ with vertices in $c$ to obtain $G^\dagger_2$.
\end{proof}



It can happen that the only critical separating cycles in a face graph $G\in \G(m,n)$ are the trivial ones, that is,  the boundary cycles of the faces of $G$.

\begin{definition}
A face graph $G$ in $\G(m,n)$ is  \emph{indivisible} if every critical separating cycle for $G$ is the boundary cycle of a face 
of $G$.
\end{definition}

In the next section it is shown how repetition of $(3,6)$-tight-preserving $TT$ edge contractions may lead to the appearance of critical separating cycles. 
Through a repeated edge contraction and cycle division process a set of terminal and indivisible face graphs may be obtained. 
Such a face graph is illustrated in Fig. \ref{f:indivis}.

\begin{center}
\begin{figure}[ht]
\centering
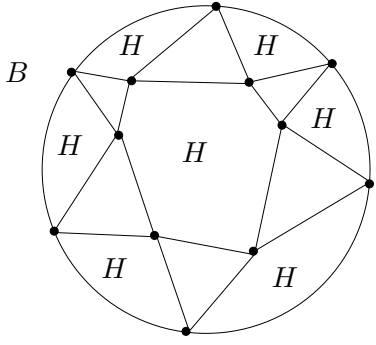
\caption{A  face graph in $\G(1,7)$ which is both terminal and indivisible.}
\label{f:indivis}
\end{figure}
\end{center}

\subsection{Key Lemmas}
If $c$ is a proper cycle in a face graph $G$, which is not the boundary of a face, then $int(c)$ denotes the subgraph of $G^\dagger$ obtained from $Int(c)$ by the removal of the edges of $c$.
The following result will be referred to as the ``hole-filling" lemma.

\begin{lemma}
\label{HoleFilling}
Let $G$ be a face graph in $\G(1,n)$.
Let $K$ be a subgraph of $G^\dagger$ and suppose that $c$ is a proper cycle  in $G\cap K$ with $E(K\cap int(c)) =\emptyset$. 
\begin{enumerate}[(i)]
\item $f(K\cup int(c))\leq f(K)$.
\item If $K$ is $(3,6)$-tight then $K\cup int(c)$ is $(3,6)$-tight.
\end{enumerate}
\end{lemma}

\begin{proof}
Since $G^\dagger$ is $(3,6)$-sparse, $f(K\cup int(c))\geq 6$ and $f(Ext(c))\geq 6$.
Note that, 
\[
6=f(G^\dagger) =f(Ext(c)) +f(int(c))-3|c|,
\]
and so 
$f(int(c))-3|c|\leq 0$. It follows that,
\[
f(K\cup int(c)) = f(K) + f(int(c))-3|c| \leq f(K).
\]
This proves $(i)$. To prove $(ii)$ apply the above argument with $f(K)=6$.
\end{proof}

The following lemma plays a key role in the proof of the main result. 

\begin{lemma}\label{l:criticalequivalence}
Let $G$ be a face graph in $\G(1,n)$ with $n\geq 1$.
Let $e$ be a contractible $TT$ edge in  $G$ with contracted face graph $G'$. 
Then the following statements are equivalent.
\begin{enumerate}[(i)]
\item $G'\notin \G(1,n)$.
\item The edge $e$ lies on a critical separating cycle of $G$.
\end{enumerate}
\end{lemma}

\begin{proof}
Suppose that $G'\notin \G(1,n)$ and let $e=uv$.
Then the discus and hole graph $(G')^\dagger$ is not $(3,6)$-tight
and so there exists a subgraph $K'$ in $(G')^\dagger$ with $f(K')\leq5$.
Let $v'$ be the vertex in $G'$ obtained by the identification of $u$ and $v$. Evidently, $v'\in V(K')$ since, otherwise, $G^\dagger$ must contain a copy of $K'$ and this contradicts the $(3,6)$-sparsity count for $G^\dagger$. 
There are two pairs of edges $xu$, $xv$ and $yu$, $yv$ in $G$ which are identified with $xv'$ and $yv'$ in $G'$ on contraction of $e$ (see Fig.~\ref{f:Kgraph}).

\begin{center}
\begin{figure}[ht]
\centering
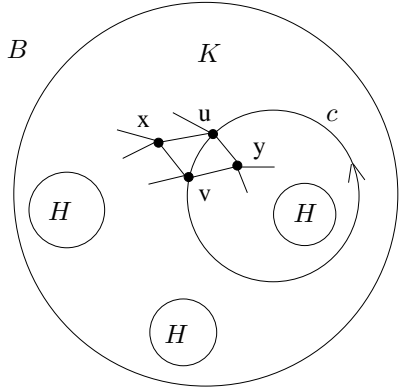
\caption{Locating a critical separating cycle for $e$.}
\label{f:Kgraph}
\end{figure}
\end{center}

Case $(a)$. Suppose that $x,y\in V (K')$. Let $K$ be the subgraph obtained from $K'$ by first adjoining the edges $xv'$ and $yv'$ to $K'$ (if necessary) and then reversing the $TT$ edge contraction on $e$. Then $f(K')\geq f(K)\geq 6$ which is a contradiction.  
 
Case $(b)$. Suppose that $x\in V (K')$ and $y\notin V(K')$.
Let $K$ be the subgraph of $G^\dagger$ obtained from $K'$ by first adjoining the edge $xv'$  to $K'$ (if necessary) and then reversing the $TT$ edge contraction on $e$. Then $f(K)\leq f(K')+1\leq 6$ and so $f(K)=6$.
In particular, $K$ is $(3,6)$-tight. Rechoose $K$, if necessary, to be a maximal $(3,6)$-tight graph in $G^\dagger$ which contains the edge $e$ and does not contain the vertex $y$.
Note that $K$ must be connected and must contain the isostatic block in $G^\dagger$.
Since $K$ is maximal, by the hole-filling lemma (Lemma \ref{HoleFilling}), $K=Ext(c)$ for some proper cycle $c$ in $G$. 
This cycle is a critical separating cycle for $G$, and so $(i)$ implies $(ii)$ in this case.

Case $(c)$. Suppose that $x\notin V(K')$ and $y\notin V(K')$. 
Let $K$ be the subgraph of $G^\dagger$ obtained from $K'$ by reversing the $TT$ edge contraction on $e$. Then $f(K)=f(K')+2\leq 7$ and so $f(K)\in\{6,7\}$.
Once again assume that $K$ is a maximal subgraph with this property. 
Then $K$ must be connected and must contain the isostatic block in $G^\dagger$.
By the planarity of $G$ there are two proper cycles $\pi_1, \pi_2$ of $G$, passing through $e$, with $int(\pi_1)$ and $int(\pi_2)$ disjoint from $K$ and containing $x$ and $y$ respectively. 
Since $K$ is maximal, by the hole-filling lemma (Lemma \ref{HoleFilling}), $K=Ext(\pi_1)\cap Ext(\pi_2)$.
Note that $f(Ext(\pi_1))\geq 6$, $f(Ext(\pi_2))\geq 6$ and 
\[6=f(G^\dagger) = f(Ext(\pi_1)) + f(Ext(\pi_2)) - f(K).\] 
Thus, since $f(K)\in\{6,7\}$, at least one of $\pi_1$ and $\pi_2$ is a critical separating cycle and so $(i)$ implies $(ii)$.

For the converse, suppose that the contractible edge $e$ lies on a critical separating cycle $c$. Then $c$ is a separating cycle for a division $G \to \{G_1, G_2\}$ and  $G_1^\dagger$ is a $(3,6)$-tight subgraph of $G^\dagger$. Since the edge $e$ lies in exactly one triangular face of $G_1^\dagger$, the graph obtained from $G_1^\dagger$ by contracting $e$ is a subgraph of $(G')^\dagger$ with freedom number $5$ and so $(i)$ does not hold.
\end{proof}

\begin{cor}\label{c:nocontraction}
Let $G$ be a face graph in $\G(1,n)$ which is both terminal and indivisible. Then $G$ contains no $TT$ edges.
\end{cor}

\begin{proof}
Suppose there exists a $TT$ edge $e$ in $G$. 
Since $G$ is terminal, either $e$ is not contractible or $e$ is contractible but the graph obtained by contracting $e$ is not in $\G(1,n)$.
If $e$ is not contractible then it must be contained in a non-facial $3$-cycle $c$.
By Lemma \ref{CritSep3Cycle}, $c$ is a critical separating cycle for $G$.
However, this contradicts the indivisibility of $G$.
If $e$ is contractible then by Lemma \ref{l:criticalequivalence},
$e$ lies on a critical separating cycle.
Again this contradicts the indivisibility of $G$ and so the result follows.
\end{proof}



\subsection{Contracting edges of $BH$ type}
A $BH$ edge $e$ of a face graph $G$ is {\em contractible} if it does not belong to any $3$-cycle in $G$. A {\em $BH$ edge contraction} is a graph move $G\to G'$ on the class of face graphs under which the vertices of a contractible $BH$ edge of $G$ are identified. 
At the level of the discus and hole graph $G^\dagger$, a contractible $BH$ edge $e$ is contained in a simplicial discus $B^\dagger$ and is an edge of exactly two $3$-cycles of $G^{\dagger}$. The contraction of $e$ preserves the freedom number of $G^\dagger$ and can be reversed by vertex splitting. 
Thus, prima facie, there is the possibility of reducing an indivisible terminal face graph with a $(3,6)$-tight discus and hole graph to a smaller face graph which also has a $(3,6)$-tight discus and hole graph.
In the case of a  block and hole graph with a single block this is always the case.

\begin{lemma}\label{l:BHreduce3}
Let $G \in \G(1,n)$, $n\geq 1$, and let $G'$ be derived from $G$ by a $BH$ edge contraction. 
Then $G'$ is a face graph in either $\G(1,n)$, $\G(1,n-1)$ or
$\G(0,0)$. 
\end{lemma}

\begin{proof}
Let $e=uv$ be the contractible $BH$ edge in $G$ with $B_1$ and $H_1$ the adjacent labelled faces of $G$ and $v'$ the vertex in $G'$ obtained on identifying of $u$ and $v$. 
Then $e$ is contained in exactly two $3$-cycles of $G^\dagger$ which lie in the simplicial discus ${B_1^\dagger}$. 
Clearly, $(G')^\dagger$ satisfies the Maxwell count since $f((G')^\dagger)=f(G^\dagger)=6$. 
The $BH$ edge contraction on $e$ reduces the length of the boundary cycle $\partial B_1$ by one. If this reduction of the boundary cycle results in a $3$-cycle then $G'$ has no $B$-labelled face. Moreover, the Maxwell count for $G'$ ensures that there are no $H$-labelled faces in $G'$. Thus $G'\in\G(0,0)$. 
If $G'$ has one $B$-labelled face then it must have either $n$ or $n-1$ $H$-labelled faces, depending on whether or not the $BH$ edge contraction on $e$ reduces the boundary cycle $\partial H_1$ to a $3$-cycle. It remains to show that $(G')^\dagger$ is $(3,6)$-sparse in this case.

If $K'$ is a subgraph of $(G')^\dagger$ then $K'$ may be obtained from a subgraph $K$ of $G^\dagger$ by the contraction of $e$.
Let $x$ and $y$ be the polar vertices of the simplicial discus $B_1^\dagger$. If $K'$ contains neither of the vertices $x, y$ then $K$ is a subgraph of $G$ with $f(K)\geq 6+(|\partial B_1|-3) + (|\partial H_1|-3)\geq 8$. Thus $f(K')=f(K)-2\geq 6$.
Suppose that $K'$ contains exactly one of the polar vertices $x, y$. Then, assuming it is the vertex $x$, it follows that $K$ is a subgraph of the triangulated sphere obtained from $G$ by substituting the simplicial disc $B_1$ with the discus hemisphere for the vertex $x$ and by inserting simplicial discs in the $H$-labelled faces of $G$. It follows that $K'$ is also a subgraph of a triangulated sphere and so $f(K')\geq 6$.
Now suppose that $K'$ contains both of the polar vertices $x, y$.
It is sufficient to consider the case when $K'$ contains the edges $xv'$ and $yv'$ and to assume that $xu,xv,yu,yv\in K$.
Then $f(K')=f(K)\geq6$.
It follows that $(G')^\dagger$ is $(3,6)$-sparse.
\end{proof}

For multiblock graphs a $BH$ edge contraction need not preserve $(3,6)$-tightness.

\begin{example}
Let $G\in \G(2,3)$ be the face graph illustrated in Fig.~\ref{f:BHprob}. Contraction of the edge $e$ leads to a vertex which is adjacent to four vertices in $\partial B_1$ and so the associated discus and hole graph is not $(3,6)$-tight.
\end{example}

\begin{center}
\begin{figure}[ht]
\centering
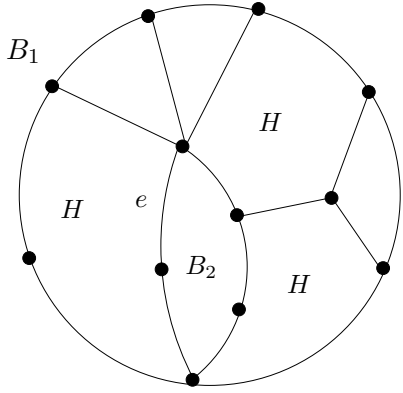
\caption{A contractible $BH$ edge $e$ in a face graph $G\in \G(2,3)$ with inadmissible contraction.}
\label{f:BHprob}
\end{figure}
\end{center}

The following analogue of Lemma \ref{l:BHreduce3} applies to multi-block graphs.

\begin{lemma}\label{l:BHreduce}
Let $G$ be a face graph in $\G(m,n)$ with $m,n\geq 1$. Let $e$ be an edge of a path in $\partial B_i\cap \partial H_j$ which has length $3$ or more
and let $G'$ be the face graph, of type $(m',n')$ obtained by the contraction of $e$. Then $G'\in \G(m',n')$.
\end{lemma}

\begin{proof}
It is clear that $(G')^\dagger$ satisfies the Maxwell count.
Let $K'$ be a subgraph of $(G')^\dagger$ with decomposition 
$K'=K_1'\cup K_2'$ such that $K_1' \subset (B_i')^\dagger$ and $E(K'_2)\cap E(B_i')=\emptyset$. We may assume that $K'$ is the contraction of
a subgraph
$K\subseteq G^\dagger$  containing $e$ and that $K$ has the corresponding
decomposition
$K=K_1\cup K_2$ (with $K_1 \subset B_i^\dagger$ and $E(K_2)\cap E(B_i^\dagger)=\emptyset$). In fact we can identify the graphs $K_2$ and $K_2'$. Note that in view of the path hypothesis the edge-less graph $K_1'\cap K_2'$ has the same number of vertices as $K_1\cap K_2$.

Observe that $K_1'$ results from $K_1$ through the loss of a vertex $v$ in $K_1$, of degree $4$ in ${G^\dagger}$.
If $v$ has degree $j$ in $K$ (and hence in $K_1$) with $1\leq j\leq 4$ then note that $f(K)\geq 6+(4-j)$ and $f(K_1)\geq 6+(4-j)$. Also, on contraction to $K_1'$ there is a reduction of
$j-1$ edges.
Thus, 
\begin{eqnarray*}
f(K')&=& f(K_1')+f(K_2')-f(K_1'\cap K_2')\\
&=& f(K_1')+f(K_2)-f(K_1\cap K_2)\\
&=& (f(K_1)-3+(j-1))+f(K_2)-f(K_1\cap K_2)\\
&=& f(K)- (4-j) \geq 6 
\end{eqnarray*}
This shows that $(G')^\dagger$ is $(3,6)$-sparse.
\end{proof}

In the light of Lemma \ref{l:BHreduce3}, the indivisible terminal face graph of Fig. \ref{f:indivis} may be reduced by $BH$ edge contractions and further edge contraction reductions become possible in view of the emerging edges of type $TT$. One can continue such reductions until termination at the terminal graph of $\G(0,0)$ which is $K_3$. In fact this kind of reduction is possible in general and forms a key part of the  proof of Theorem \ref{t:oneblockA}.



\begin{definition}
A face graph $G$ is {\em $BH$-reduced} if it contains no contractible $BH$ edges.
\end{definition}

\begin{cor}
\label{Facetheorem}
For each $n\geq 1$, there is no face graph in $\G(1,n)$ which is terminal, indivisible and $BH$-reduced.
\end{cor}

\proof
Suppose there exists $G\in\G(1,n)$ which is terminal, indivisible and $BH$-reduced.
By Corollary \ref{c:nocontraction},  $G$ contains no $TT$ edges.
If an edge $e$ in $G$ is of type $BH$ then, since $G$ is $BH$-reduced, $e$ is not contractible and so must be contained in a non-facial $3$-cycle $c$ of $G$.
By Lemma \ref{CritSep3Cycle}, $c$ is a critical separating cycle for $G$. 
However, this contradicts the assumption that $G$ is indivisible and so $G$ contains no $BH$ edges.
By Proposition \ref{FaceProp}, $G$ contains a critical separating cycle for $G$ which is not the boundary of a face.
However, this contradicts the indivisibility of $G$ and so there can be no face graph in $\G(1,n)$ which is terminal, indivisible and $BH$-reduced.
\endproof


\begin{cor}
\label{c:canconstruct} 
Let $G$ be a face graph in $\G(1,n)$. 
Then there exists a rooted tree in which each node is labelled by a face graph such that,
\begin{enumerate}[(i)]
\item the root node is labelled $G$,
\item every node has either one child which is obtained from its parent node by a $TT$ or $BH$ edge contraction, or, two children which are obtained from their parent node by a critical separating cycle division, 
\item each node is either contained in $\G(1,m)$ for some $m\leq n$  and is not a leaf, or, is contained in $\G(0,0)$ (in which case it is a leaf). 
\end{enumerate}
\end{cor}

\proof
The statement follows by applying Corollary \ref{Facetheorem} together with Lemma \ref{CycleDivision} and Lemma \ref{l:BHreduce3}.
\endproof

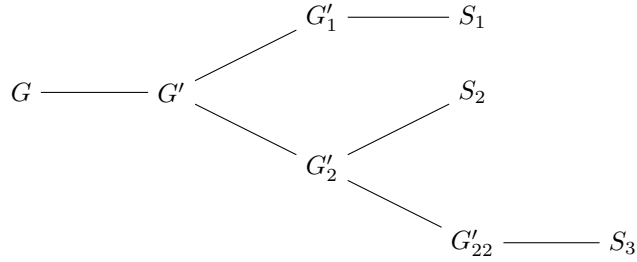
\begin{figure}[ht]
\centering
  \begin{tabular}{  c   }
\begin{tikzpicture}[level distance=2cm,
  level 1/.style={sibling distance=2cm},
  level 2/.style={sibling distance=2cm},
	level 3/.style={sibling distance=2cm},
	level 4/.style={sibling distance=2cm}]
  \node {$G$} [grow=right]
    child {node {$G'$}
      child {node {$G_2'$}
				child{node {$G'_{22}$}
				child{node {$S_3$}}
				}
				child{node {$S_2$}}
			}
      child {node {$G_1'$}
				child{node {$S_1$}}
      }};
\end{tikzpicture}
\end{tabular}

 \caption{Deconstructing a face graph $G\in\G(1,n)$. Each node is obtained from its parent by a $TT$ or $BH$ edge contraction, or, by a critical separating cycle division. Each leaf is contained in $\G(0,0)$.}
\label{ConstructionA}
\end{figure}

In the case of general block and hole graphs one can also perform division at critical cycles, and there are counterparts to Lemma \ref{l:criticalequivalence} and Corollary \ref{c:nocontraction}.
However, as the following example shows, there are face graphs in $\G(m,n)$, $m\geq2$, which are terminal, indivisible and $BH$-reduced.

\begin{example}
Fig.~\ref{f:tnet} shows a face graph $G\in\G(2,6)$ which is terminal, indivisible and $BH$-reduced. Note that the associated  block and hole graphs $\hat{G}$ are $3$-rigid. This follows from the fact that they are constructible from $K_3$ by vertex splitting together with Henneberg degree $3$ and degree $4$ vertex extension moves. 
\end{example}

\begin{center}
\begin{figure}[ht]
\centering
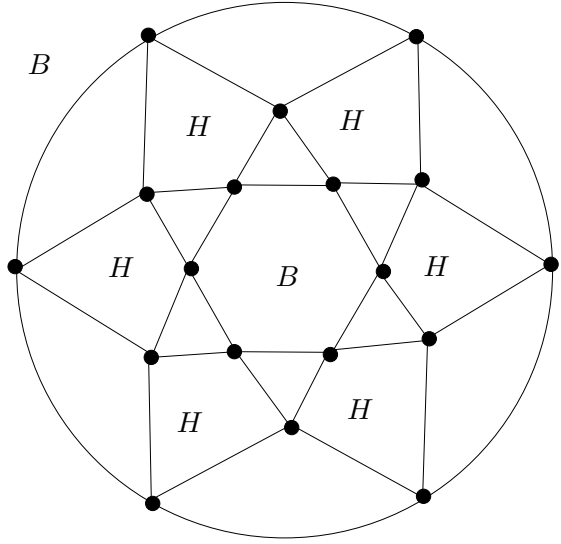
\caption{A face graph in $\G(2,6)$ which is terminal, indivisible and $BH$-reduced.}
\label{f:tnet}
\end{figure}
\end{center}

\subsection{Generic rigidity of block and hole graphs}
\label{ProofMaintheorem}
Let $J$ be a simple graph and let $v$ be a vertex of $J$ with adjacent vertices $v_1,v_2,\ldots, v_n$, $n\geq 2$.
Construct a new graph $\tilde{J}$ from $J$ by,
\begin{enumerate}
\item removing the vertex $v$ and its incident edges from $J$,
\item adjoining two new vertices $w_1,w_2$,
\item adjoining the edge $w_1v_j$ or the edge $w_2v_j$ for each 
$j=3,4,\ldots,n$.
\item adjoining the five edges $v_1w_1,v_2w_1$, $v_1w_2,v_2w_2$ and $w_1w_2$.
\end{enumerate}
The graph move $J\to \tilde{J}$ is called  {\em vertex splitting}. It is shown in \cite{whi-vertex} that if $J$ is minimally $3$-rigid then so too is $\tilde{J}$. (See also the Appendix).






\begin{theorem}\label{t:oneblockA}  
Let $\hat{G}$ be a block and hole graph with a single block. 
Then the following statements are equivalent.
\begin{enumerate}[(i)]
\item $\hat{G}$ is minimally $3$-rigid.
\item $\hat{G}$ is $(3,6)$-tight.
\item $\hat{G}$ is constructible from $K_3$ by the moves of vertex splitting and isostatic block substitution.
\end{enumerate}
\end{theorem}

\proof
The implication $(i)\Rightarrow (ii)$ is well known for general
minimally $3$-rigid graphs. The implication $(iii)\Rightarrow(i)$ follows from the isostatic block substitution principle (Lemma \ref{l:graphequivalentsB}) and the fact that vertex splitting preserves minimal $3$-rigidity (see Appendix).

To prove $(ii)\Rightarrow (iii)$, apply the following induction argument based on the number of vertices of the underlying face graph. 
Let $P(k)$ be the statement that every $(3,6)$-tight block and hole graph $\hat{G}$ with a single block and $|V(G)|=k$ is constructible from $K_3$ by the moves of vertex splitting and isostatic block substitution. Note that if $|V(G)|=4$ then $G$ is a $4$-cycle with one $B$-labelled face and one $H$-labelled face. In this case, every block and hole graph $\hat{G}$ is clearly constructible from $K_3$ by applying a single vertex splitting move to obtain the minimally $3$-rigid graph $K_4$ and then substituting this $K_4$ with the required isostatic block for $\hat{G}$. 
Thus the statement $P(4)$ is true and this establishes the base of the induction.

Now assume that the statement $P(k)$ holds for all $k=4,5,\ldots,l-1$ where $l\geq 5$. Let $\hat{G}$ be a $(3,6)$-tight block and hole graph with a single block and $|V(G)|=l$. By Corollary \ref{l:graphequivalentsA}, the discus and hole graph $G^{\dagger}$ is also $(3,6)$-tight and so $G\in\G(1,n)$ for some $n$.
Thus $G$ admits a $TT$ edge contraction, a $BH$ edge contraction or a critical separating cycle division as described in the reduction scheme for face graphs in $\G(1,n)$ (Corollary \ref{c:canconstruct}). 
In the case of a $TT$ or $BH$ edge contraction $G\to G'$, the contracted face graph $G'$ has fewer vertices than $G$ and is contained in either $\G(1,m)$ for some $m\leq n$, or, in $\G(0,0)$. 
In the former case, the induction hypothesis implies that $(G')^\dagger$ is constructible from $K_3$ by the moves of vertex splitting and isostatic block substitution. In the latter case, $(G')^\dagger$ is the graph of a triangulated sphere and so is constructible from $K_3$ by vertex splitting alone (see Appendix).
It follows that $G^\dagger$ is itself constructible from $K_3$ by  vertex splitting and isostatic block substitution.  
In the case of a critical separating cycle division $G\to \{G_1,G_2\}$, 
$G$ is obtained from two face graphs $G_1$ and $G_2$, each with fewer vertices than $G$. Moreover, for each $j=1,2$ either $G_j\in \G(1,m_j)$ for some $m_j\leq n$, or, $G_j\in \G(0,0)$.
Thus it again follows that both $G_1^\dagger$ and $G_2^\dagger$ are constructible from $K_3$ by vertex splitting and isostatic block substitution. Note that $G_1^\dagger$ is minimally $3$-rigid and so  may be used as a substitute for the isostatic block of $G_2^\dagger$. In this way $G^\dagger$ is shown to be constructible from $K_3$ in the required manner. 
This establishes the inductive step and so the proof of the implication $(ii)\Rightarrow (iii)$ is complete.
\endproof

\section{Girth inequalities}
\label{Sect_GI}
We now examine certain cycle length inequalities for  block and hole graphs that were considered in Finbow-Singh and Whiteley \cite{fin-whi}. Recall from Ex.~\ref{Ex2} that $G^\circ$ denotes the block and whole graph obtained from a face graph $G$ by adjoining $2(|\partial B|-3)$ edges to each $B$-labelled face so that each isostatic block $B^\circ$ is the graph of a triangulated sphere.

\subsection{Index of a collection of labelled faces} 
Let $\B'$ and $\H'$ respectively be collections of $B$-labelled and $H$-labelled faces of a face graph $G$. 
The {\em index} of the collection $\B'\cup \H'$ is defined as,
\[\ind(\B'\cup\H') =   \sum_{B\in \B'} (|\partial B|-3)- \sum_{H\in \H'} (|\partial H|-3)\]

\begin{lemma}
\label{BHTight}
Let $G=S(\B,\H)$ be a face graph of type $(m,n)$.
\begin{enumerate}[(i)]
\item If $\C$ and $\C'$ are two collections of labelled faces of $G$ then, 
\[\ind(\C\cup\C') = \ind(\C)+\ind(\C') - \ind(\C\cap\C').\]
\item $f(G^\circ) = 6-\ind(\B\cup\H)$.
\item If $G^\circ$ satisfies the Maxwell count then, 
\[\ind((\B\cup\H)\backslash \C)=-\ind(\C)\] 
for each collection $\C$ of labelled faces of $G$.
\end{enumerate}
\end{lemma}

\proof
$(i)$ This follows by simple counting.

$(ii)$ The face graph $G$ is obtained from the graph of a triangulated sphere $S$. By construction, 
\[|E(G^\circ)|=|E(S)| +\ind(\B\cup\H).\]
Moreover, $S$ and $G^\circ$ have the same vertex set and so,
\[f(G^\circ) = 3|V(G^\circ)|-|E(G^\circ)|=f(S) - \ind(\B\cup\H).\]
The graph of a triangulated sphere $S$ must satisfy the Maxwell count  and so the result follows.

$(iii)$
Let $\C$ be a collection of labelled faces of $G$. 
By $(i)$, 
\[\ind(\B\cup\H) = \ind(\C)+\ind((\B\cup\H)\backslash \C).\]
If $G^\circ$ satisfies the Maxwell count then, by $(ii)$, $\ind(\B\cup\H)=0$ and so the result follows. 
\endproof

\begin{definition}
A face graph $G$ is said to satisfy the {\em girth inequalities} if, for every proper cycle $c$ in $G$, and every planar realisation of $G$,
\[
|c|\geq  |\ind(\C)|+3\]
where $\C$ is the collection of $B$-labelled and $H$-labelled faces of $G$ which lie inside $c$.
\end{definition}

A block and hole graph $\hat{G}$ satisfies the girth inequalities if it is derived from a face graph $G$ which satisfies the girth inequalities.

\begin{example}
Let $G$ be a face graph of type $(1,1)$, so that $G$ has exactly one $B$-labelled face and exactly one $H$-labelled face. 
Then $G$ satisfies the girth inequalities if and only if the lengths of the boundaries of the $B$-labelled face and the $H$-labelled face are equal and, letting $r$ denote this common boundary length, 
every proper cycle in $G$ which winds around $H$ has length at least $r$. 
\end{example}

\begin{lemma}
\label{GILemmaA}
Let $G$ be a face graph of type $(m,n)$.
If $G$ satisfies the girth inequalities then $G^\circ$ satisfies the Maxwell count.
\end{lemma}

\begin{proof} 
By Lemma \ref{BHTight}$(ii)$ it is sufficient to show that $\ind(\B\cup \H)=0$.
Choose any $H$-labelled face $H_1$ in $G$ and let $\C=(\B\cup\H)\backslash\{H_1\}$. 
Applying the girth inequalities,
\[\ind(\B\cup\H) = \ind(\C)  - (|\partial H_1|-3)  
\leq |\ind(\C)|  - (|\partial H_1|-3)  
\leq 0.\]
To obtain the reverse inequality, choose any $B$-labelled face $B_1$ in $G$ and let $\C'=(\B\cup\H)\backslash\{B_1\}$. 
By the girth inequalities, 
\[\ind(\B\cup\H) = (|\partial B_1|-3) + \ind(\C')  
\geq |\ind(\C')| + \ind(\C') 
\geq 0.\]
\end{proof}

\begin{prop}
\label{l:girth}
Let $c$ be a proper cycle in a face graph $G$ of type $(m,n)$
and let  $\C$ be a  collection of labelled faces of $G$ which lie inside $c$ for some planar realisation of $G$.
\begin{enumerate}[(i)]
\item If $G^\circ$ is simple and $(3,6)$-sparse then $|c|\geq \ind(\C)+3$.
\item If $G^\circ$ is simple and $(3,6)$-tight  then  $|c|\geq |\ind(\C)\,|+3$.
\end{enumerate}
In particular, if $G^\circ$ is simple and $(3,6)$-tight then $G$ satisfies  the girth inequalities.
\end{prop}

\proof
Let $S$ be the graph of a triangulated sphere and let $c$ be a proper cycle of edges of length greater than $3$.
Then $c$ determines two simplicial  discs $D_1$ and $D_2$ with intersection equal to $c$. Since each simplicial disc may be completed to the graph of a triangulated sphere by the addition of $|c|-3$ edges it  follows that,
\[
f(D_1)=f(D_2) = 6+(|c|-3).
\]
Suppose a graph $K_1$ is derived from $D_1$ by keeping the same vertex set and subtracting and adding various edges. Then $K_1$ will fail the sparsity count $f(K_1)\geq 6$ if the total change in the number of edges is an increase by more than $|c|-3$ edges.

Consider now the face graph $G$ and suppose it is derived from the graph of a triangulated sphere $S$.
Fix a planar representation of $G$ and let $c$ be a proper cycle in $G$. As in the previous paragraph, $c$ determines two simplicial discs $D_1$ and $D_2$ in $S$. Without loss of generality, assume that $D_1$ contains the edges of $S$ which lie inside $c$ and $D_2$ contains the edges which lie outside $c$. Let $K_1$ and $K_2$ be the corresponding subgraphs of the block and hole graph $G^\circ$. Thus $K_1$ and $K_2$ are derived from $D_1$ and $D_2$ respectively by removing edges which correspond to $H$-labelled faces in $G$ and adjoining the edges of each isostatic block.

$(i)$ If $G^\circ$ is $(3,6)$-sparse then $f(K_1)\geq 6$.
Thus the total change in the number of edges in deriving $K_1$ from $D_1$ does not exceed $|c|-3$ in magnitude.
This implies the inequality $|c|-3\geq \ind(\C)$.

$(ii)$
Applying the argument for $(i)$ to $K_2$, $f(K_2)\geq 6$ and so the total change in the number of edges in deriving $K_2$ from $D_2$ does not exceed $|c|-3$. Thus,
\[|c|-3\geq \ind((\B\cup\H)\backslash \C).\]
By Lemma \ref{BHTight}, $\ind((\B\cup\H)\backslash \C) ) = - \ind(\C)$ and so $|c|-3\geq |\ind(\C)|$.
\endproof








\subsection{Critical girth cycles}


\begin{definition}
A proper cycle $c$ in a face graph $G$ is called a \emph{critical girth cycle} for $G$ if, for some planar realisation of $G$,
\[
|c| =  |\ind(\C)|+3
\]
where $\C$ is the collection of $B$-labelled and $H$-labelled faces of $G$ which lie inside $c$.
\end{definition}

Recall from Def.~\ref{d:critcycle} the definition of a critical separating cycle for a face graph.

\begin{lemma}\label{l:critequivalence}
Let $G$ be a face graph of type $(m,n)$ and suppose the block and hole graphs for $G$ satisfy the Maxwell count.
If $c$ is a proper cycle in $G$ then the following statements are equivalent.
\begin{enumerate}[(i)]
\item $c$ is a critical girth cycle for $G$.
\item Either $Ext(c)$ or $Int(c)$ satisfies the Maxwell count.
\end{enumerate}
In particular, if $G\in \G(m,n)$ then $c$ is a critical girth cycle if and only if it is a critical separating cycle. 
\end{lemma}

\begin{proof}
Fix a planar realisation for $G$ and let $\B'\cup\H'$ be the labelled faces of $G$ which lie inside $c$. Let $G_1$ be the face graph obtained from $G$ by removing  edges and vertices which are interior to $c$ and, if $|c|\geq4$, labelling the face with boundary $c$ by $H$. 
Then $f(G_1^\circ) = f(G^\circ)-\ind(\B'\cup\H') + (|c|-3)$.
It follows that $G_1^\circ$ satisfies the Maxwell count if and only if $|c|=\ind(\B'\cup\H')+3$.
Similarly, let $G_2$ be the face graph obtained from $G$ by removing  edges and vertices which are exterior to $c$ and, if $|c|\geq4$, labelling the face with boundary $c$ by $H$. 
Then, by Lemma \ref{GILemmaA}$(iii)$, 
$G_2^\circ$ satisfies the Maxwell count if and only if $|c|=-\ind(\B'\cup\H')+3$. Thus, $c$ is a critical girth cycle if and only if either $G_1^\circ$ or $G_2^\circ$ satisfies the Maxwell count. The result now follows from Corollary \ref{l:graphequivalentsA}.
\end{proof}

\subsection{One block and $n$ holes}
From the arguments of \cite{fin-whi} it follows that a block and hole graph with a single block and a single hole is $(3,6)$-tight if and only if the underlying face graph satisfies the girth inequalities.
In Theorem \ref{l:girth1block} this equivalence is extended to the case of block and hole graphs with a single block and $n$ holes for any $n\geq 1$.


\begin{lemma}
\label{GILemmaB}
Let $G\to G'$ be a $TT$ edge contraction or a $BH$ edge contraction on a face graph $G$ of type $(1,n)$.
If $G$ satisfies the girth inequalities and contains no critical girth cycles, other than boundary cycles, then $G'$ satisfies the girth inequalities.
\end{lemma}

\proof
If $G'$ is obtained from $G$ by contracting a $TT$ edge $e$
then this contraction does not alter the boundary of any labelled face of $G$. 
If $G'$ is obtained from $G$ by contracting a $BH$ edge $e$ then this contraction reduces by one the boundary lengths of the $B$-labelled face and some $H$-labelled face $H_1$. 
All other labelled faces of $G$ are unchanged. 
Let $c'$ be a proper cycle in $G'$. Then there is a proper cycle $c$ in $G$ such that either $c=c'$, or, $c'$ is obtained from $c$ by contracting the edge $e$. 
If $e$ is an edge of $c$ then $B_1$ and $H_1$ must lie in complementary regions of the complement of $c$.
Thus the index of the exterior and interior labelled faces for $c$ are, respectively, reduced and increased by one.
If $e$ is not an edge of $c$ then the $B$ and $H$ labelled faces both lie either inside or outside $c$.
Thus the index of the exterior and interior labelled faces for $c$ are unchanged. 
Since $c$ is not a critical girth cycle in $G$, in each of these cases the girth inequality is satisfied by $c'$.
\endproof

\begin{lemma}
\label{GILemmaC}
Let $G$ be a face graph of type $(1,n)$ and let $G\to \{G_1,G_2\}$ be a separating cycle division on a critical girth cycle $c$ in $G$.
If $G$ satisfies the girth inequalities then $G_1$ and $G_2$ both satisfy the girth inequalities.
\end{lemma}

\proof
Let $\C$ denote the collection of labelled faces of $G$ which lie inside $c$.
Evidently, $\ind(\C)\leq 0$ and so, since $c$ is a critical girth cycle in $G$, $|c|-3= -\ind(\C)$. Moreover, by Lemma \ref{GILemmaA}, $G^\circ$ satisfies the Maxwell count and so, by Lemma \ref{BHTight}, $|c|-3 = \ind((\B\cup\H)\backslash\C)$. 
If $c_1$ is a proper cycle in $G_1$ then $c_1$ is also a proper cycle in $G$.
Let $\D$ denote the collection of labelled faces of $G$ which lie inside $c_1$ and let $\C_1$ denote the collection of labelled faces of $G_1$ which lie inside $c_1$.
Since $|c|-3= -\ind(\C)$, it follows that $\ind(\D) = \ind(\C_1)$.
Since $G$ satisfies the girth inequalities, 
$|c_1|\geq |\ind(\D)|+3=|\ind(\C_1)|+3$.
If $\C_1'$ denotes the labelled faces of $G_1$ which lie 
outside $c_1$ then, again since $|c|-3= -\ind(\C)$, it follows that $\ind(\C_1') = -\ind(\D)$.
Thus,  $|c_1|\geq |\ind(\C_1')|+3$ and so $G_1$ satisfies the girth inequalities.
Similarly, if $c_2$ is a proper cycle in $G_2$ then $c_2$ is also a proper cycle in $G$ and, since $|c|-3 = \ind((\B\cup\H)\backslash\C)$,
it follows that 
$G_2$ satisfies the girth inequalities.

\endproof





The following theorem completes the proof of Theorem \ref{t:oneblock} in the single block case.

\begin{theorem}\label{l:girth1block}
Let $\hat{G}$ be a block and hole graph with a single block. 
Then the following are equivalent.
\begin{enumerate}[(i)]
\item $\hat{G}$ is minimally $3$-rigid.
\item $\hat{G}$ satisfies the girth inequalities.
\end{enumerate}
\end{theorem}

\begin{proof}
If $\hat{G}$ is minimally $3$-rigid then, by the isostatic block substitution principle, Lemma \ref{l:graphequivalentsB}, $G^\circ$ is minimally $3$-rigid for any choice of triangulated sphere $B^\circ$. 
In particular, $G^\circ$ is 
$(3,6)$-tight and so, by Proposition \ref{l:girth}, $G$ satisfies the girth inequalities.

To prove the converse, apply the following induction argument.
Let $P(k)$ be the statement that every block and hole graph $\hat{G}$ with a single block which satisfies the girth inequalities and has $|V(G)|=k$, is minimally $3$-rigid.
The statement $P(4)$ is true since in this case there exists only one face graph $G$, namely a $4$-cycle with one $B$-labelled face and one $H$-labelled face. Clearly, $G$ satisfies the girth inequalities and has minimally $3$-rigid block and hole graphs.
This establishes the base of the induction.

Suppose that $P(k)$ is true for all $k=4,5,\ldots,l-1$ and let  $\hat{G}$ be a block and hole graph with a single block which satisfies the girth inequalities and has $|V(G)|=l$.
Note that, by Lemma \ref{GILemmaA}, each block and hole graph $G^\circ$ satisfies the Maxwell count. 
If $G$ contains a critical girth cycle $c$, which is not the boundary of a face, then by Lemma \ref{GILemmaC}
the face graphs $G_1$ and $G_2$ obtained by separating cycle division on $c$ both satisfy the girth inequalities. 
Note that $G_1$ and $G_2$ are each either face graphs with a single $B$-labelled face and fewer vertices than $G$, or, are triangulations of a triangle. 
It follows that both $G_1$ and $G_2$ have minimally $3$-rigid block and hole graphs. By the block substitution principle (Lemma \ref{l:graphequivalentsB}) the isostatic block of $G_2^\dagger$ may be substituted with $G_1^\dagger$ to obtain $G^\dagger$. Thus $G$ has minimally $3$-rigid block and hole graphs. 

Now suppose that there are no critical girth cycles in $G$, other than the boundary cycles of faces of $G$. 
If $G$ contains no edges of type $TT$ or $BH$ then, by Proposition \ref{FaceProp}, $G$ contains a proper cycle $\pi$, which is not the boundary of a face, such that $Ext(\pi)$ satisfies the Maxwell count.
By Lemma \ref{l:critequivalence}, $\pi$ is a critical girth cycle for $G$.
This is a contradiction and so $G$ must contain an edge of type $TT$ or $BH$. Moreover, such an edge must be contractible since any non-facial $3$-cycle would be a critical girth cycle for $G$.

Suppose a face graph $G'$ is obtained from $G$ by contracting a $TT$ or a $BH$ edge $e$. 
Then $G'$ is either a face graph with a single $B$-labelled face and fewer vertices than $G$, or, is a triangulation of a triangle.  By Lemma \ref{GILemmaB}, $G'$ satisfies the girth inequalities and so
$G'$ must have minimally $3$-rigid block and hole graphs. 
Now $G^\dagger$ may be obtained from $(G')^\dagger$ by vertex splitting and so $G$ also has minimally $3$-rigid block and hole graphs.
This establishes that the statement $P(l)$ is true and so, by the principle of induction, the theorem is proved.
\end{proof}

\begin{remark}
The girth inequalities give an efficient condition for the determination of generic isostaticity. Consider, for example, the structure graph of a  \emph{grounded geodesic dome} in which  a number of edges have been removed. By such a dome we mean, informally, a triangulated bar-joint framework with a ``fairly uniform" distribution of joints lying on a subset of a sphere determined by a half space and where the  edges are of  length less than a ``small multiple" (say two) of the minimum separation distance between framework joints. Also,
on the intersection of the plane and the sphere there is a cycle subgraph  of ``grounded" joints. If the windows arising from an edge depletion are sparsely positioned then it can be immediately evident that the girth inequalities prevail.
\end{remark}

In \cite{fin-whi} the following theorem is  obtained.

\begin{theorem}
Let $\hat{G}$ be a block and hole graph with one block and one hole
such that $|\partial B|=|\partial H|=r$.
If there exist $r$ vertex disjoint paths in $G$ which include the vertices of the labelled faces then $\hat{G}$ is $3$-rigid.
\end{theorem}

We note that this also follows from Theorem \ref{l:girth1block}.
 Indeed if the disjoint path condition holds then it is evident that every cycle $c$ associated with the single hole
has length at least $r$ since it must cross each of the $r$ paths. Thus the girth inequalities hold.
Similarly, Conjecture \ref{Conj_WW} follows on verifying that the $5$-connectedness condition ensures that the girth inequalities hold.

\subsection{Separation conditions}
The following separation conditions for block and hole graphs $\hat{G}$ were indicated in  \cite{fin-whi} (see Conjecture 5.1 and Proposition 5.4) and are necessary conditions for minimal $3$-rigidity.

\begin{cor}
\label{Separation}
Let $\hat{G}$ be a minimally $3$-rigid block and hole graph with face graph $G$ of type $(m,n)$.
\begin{enumerate}[(i)]
\item There are no edges in $G$ between nonadjacent vertices in the boundary of a labelled face of $G$. 
\item  
Each pair of $H$-labelled faces in $G$ share at most two vertices and these vertices must be adjacent.
\end{enumerate}
\end{cor}

\proof
$(i)$ If there exists an edge between two nonadjacent vertices in the boundary of a labelled face of $G$ then there exists a cycle in $G$ which violates the girth inequalities.  

$(ii)$ If two $H$-labelled faces in $G$ share more than two vertices
then by the girth inequalities there exists a $B$-labelled face within their joint perimeter cycle. 
However, this implies that the block and hole graphs for $G$ fail to be $3$-connected.
Similarly, if two $H$-labelled faces in $G$ share two nonadjacent vertices then the block and hole graphs for $G$ fail to be $3$-connected.

  
\endproof

The following example shows that Conjectures 5.1 and 5.2 of  \cite{fin-whi} are not true in general. 

\begin{example}
\label{CounterEx}
Let $G$ be the face graph of type $(2,2)$ with planar realisation illustrated in Fig.~\ref{f:Phatgraph}.
The block and hole graph $G^\circ$ satisfies the separation conditions
of Corollary \ref{Separation} (and of \cite{fin-whi}).
Also, $G^\circ$ is $(3,6)$-tight and, by Lemma \ref{GILemmaA}, 
$G$ satisfies the girth inequalities. 
However, $G^\circ$ is not minimally $3$-rigid since it may be reduced to a $2$-connected graph by inverse Henneberg moves on vertices of degree $3$. 
\end{example}

\begin{center}
\begin{figure}[ht]
\centering
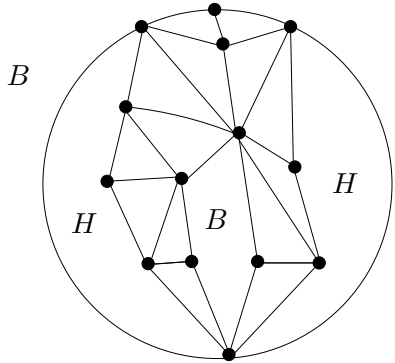
\caption{A face graph of type $(2,2)$ which satisfies the girth inequalities and separation conditions but does not have a $3$-rigid
block and hole graph.}
\label{f:Phatgraph}
\end{figure}
\end{center}

\subsection{Block-hole transposition}
\label{BHTrans}
We next observe that the characterisation of minimally $3$-rigid block and hole graphs with a single block also provides a characterisation in the single hole case.
Let $G_t$ be the face graph  obtained from graph $G$ by replacing $B$ labels by $H$ labels  and $H$ labels by $B$ labels.

\begin{cor}\label{t:onehole}  Let 
$\hat{G}$ be a block and hole graph with a single hole. 
Then the following are equivalent.
\begin{enumerate}[(i)]
\item   $\hat{G}$ is minimally $3$-rigid.
\item $\hat{G}$ is $(3,6)$-tight.
\item  $\hat{G}$ is constructible from $K_3$ by vertex splitting and isostatic block substitution.
\item $\hat{G}$ satisfies the girth inequalities.
\end{enumerate}
In particular, ${G}^\dagger$ is minimally $3$-rigid if and only if ${G_t}^\dagger$ is minimally $3$-rigid.
\end{cor}

\begin{proof}
The implications $(iii)\implies (i)\Rightarrow (ii)\Rightarrow (iv)$ have already been established more generally for face graphs of type $(m,n)$.
If $G$ satisfies the girth inequalities then 
 $G_t$ also satisfies the girth inequalities and so there exists a reduction scheme for $G_t$ as described in Corollary \ref{c:canconstruct}. This same reduction scheme may be applied to show that the block and hole graphs for $G$ are minimally $3$-rigid.
Thus the equivalence of $(i)-(iv)$ is established.
The final statement follows since $G$ satisfies the girth inequalities if and only if $G_t$ satisfies the girth inequalities.
\end{proof}

\section{Appendix}
A bar-joint framework in  $\bR^3$ consists of a simple graph $G=(V,E)$ and a placement $p:V\to \bR^3$, such that $p(v)\not=p(w)$ for each edge $vw\in E$. 
An {\em infinitesimal flex} of $(G,p)$ is an assignment $u:V\to \bR^3$ which satisfies the infinitesimal flex condition
$(u(v)-u(w))\cdot (p(v)-p(w))=0$ for every edge $vw\in E$. 
A \emph{trivial infinitesimal flex} of $(G,p)$ is one which extends to an infinitesimal flex of any containing framework, which is to say that it is a linear combination of a translation infinitesimal flex and a rotation infinitesimal flex. The framework $(G,p)$ is \emph{infinitesimally rigid} if the only infinitesimal flexes are trivial and the graph $G$ is $3$-rigid if every generic framework $(G,p)$ is infinitesimally rigid. See \cite{gra-ser-ser}.

\subsection{Vertex splitting}The proof of rigidity preservation under vertex splitting indicated in Whiteley \cite{whi-vertex} is based on static self-stresses and $3$-frames. For completeness we give an 
infinitesimal flex proof of this important result.

Let $G=(V,E)$ with $v_1,v_2,\dots ,v_r$ the vertices of $V$ and $v_1v_2, v_1v_3, v_1v_4$ edges in $E$. Let $G'=(V',E')$ arise from a vertex splitting move on $v_1$ which introduces the new vertex $v_0$ and the new edges $v_0v_1, v_0v_2, v_0v_3$. 
Some of the remaining edges $v_1v_t$ may be replaced by the edges $v_0v_t$. 
Let $p:V\to \bR^3$ be a generic realisation with $p(v_i)=p_i$
and for $n=1,2, \dots $ let $q^{(n)}:V'\to \bR^3$ be  nongeneric realisations which extend $p$, where $q^{(n)}(v_0)= q_0^{n}, n=1,2, \dots $ is a sequence of points on the line segment from $p_1$ to $p_4$ which converges to $p_1$. 
\begin{center}
\begin{figure}[ht]
\centering
\scalebox{0.8}{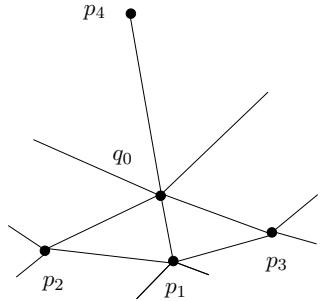}
\caption{The vertex $q_0 = q^{(1)}(v_0)$.}
\label{f:vertexsplit}
\end{figure}
\end{center}
Let $u^{(n)}, n=1,2, \dots ,$ be infinitesimal flexes of $(G',q^{(n)}), n=1,2, \dots ,$ which are of unit norm in $\bR^{3(r+1)}$. By taking a subsequence we may assume that $u^{(n)}$ converges to an infinitesimal flex $u^{(\infty)}$ of the degenerate realisation of
$G'$ with $q(v_0)=q(v_1)=q_1$. In view of the line segment condition we have, 
\[
u^{(n)}_0\cdot(p_0^{(n)}-p_4)=u^{(n)}_1\cdot(p_1-p_4).
\] 
for each $n$. Also we have,
\[
u^{(n)}_0\cdot(p_0^{(n)}-p_2)=u^{(n)}_1\cdot(p_1-p_2),\quad u^{(n)}_0\cdot(p_0^{(n)}-p_3)=u^{(n)}_1\cdot(p_1-p_3),
\]
and it follows from the generic position of $p_2,p_3$ and $p_4$ that
$u^{(\infty)}_0=u^{(\infty)}_1$. Thus $u^{(\infty)}$ restricts to an infinitesimal flex $u$  of $(G,p)$. Note that the norm of $u$ is nonzero.

We now use the general construction of the limit flex in the previous paragraph to show that if $G'$ is not $3$-rigid then neither is $G$.
Indeed if  $G'$ is not $3$-rigid then there exists a sequence  as above in which each flex $u^{(n)}$ is orthogonal in $\bR^{3(r+1)}$
to the space of trivial infinitesimal flexes. It follows that
 $u^{(\infty)}$ is similarly orthogonal and that the restriction flex $u$ of $(G,p)$ is orthogonal in  $\bR^{3r}$
to the space of trivial infinitesimal flexes. Since $u$ is nonzero
$G$ is not $3$-rigid, as desired.

\subsection{A proof of Gluck's theorem}
\label{Gluck}
In our terminology Gluck's theorem (\cite{glu}) asserts that the (unlabelled) face graphs $G$ of type $(0,0)$ are $3$-rigid. For convenience we give a direct proof here.
In view of $3$-rigidity preservation under vertex splitting it will be enough to show that $G$ derives from $K_3$ by a sequence of vertex splitting moves. 

Suppose that $G$ is not the result of a planar vertex splitting move on a face graph of type $(0,0)$. We show that $G=K_3$. Suppose that $G$  has the minimum number of vertices amongst all such graphs and suppose also, by way of contradiction, that $G \neq K_3$. Every edge of $G$ is of type $TT$ and we may consider an edge $e=uv$ with associated edges $xu, xv$ and $yu,yv$ for its adjacent faces. Since $G$ is minimal the contraction of $G$ under $e$ cannot be a simple graph and so there is a nonfacial triangle with edges $zu,zv$ and $uv$.  But now the subgraph consisting of the $3$-cycle $z,u,v$ and its interior is of type $(0,0)$ and is a smaller graph than $G$. It is not equal to $K_3$, since it contains $x$ or $y$, and so it has an edge contraction. Thus $G$ itself has an edge contraction to a simple graph, a contradiction.

\bibliographystyle{amsplain}

\end{document}